\documentclass[11pt,a4paper]{article}

\usepackage[utf8]{inputenc}
\usepackage[T1]{fontenc}
\usepackage[english]{babel}
\usepackage{amsmath}
\usepackage{amssymb}
\usepackage{amsthm}
\usepackage{sansmath}
\usepackage{mathtools}
\usepackage[titletoc,toc,title]{appendix}

\usepackage{ae}
\usepackage{icomma}
\usepackage{units}
\usepackage{color}
\usepackage{graphicx}
\usepackage{caption}
\usepackage{subcaption}
\usepackage{bbm}
\usepackage[square, numbers, sort]{natbib}
\usepackage{multirow}
\usepackage{array}
\usepackage{geometry}
\usepackage{fancyhdr}
\usepackage{fncychap}
\usepackage[hyphens]{url}
\usepackage[pdfpagelabels=false]{hyperref}
\usepackage{lettrine}
\usepackage{theoremref}
\usepackage[]{todonotes}

\newcommand{\bi}{\mathsf{BI}} % brownian interlacements
\newcommand{\be}{\mathsf{BE}} % brownian excursions
\newcommand{\geo}{\overset{g}{\leftrightarrow}} %geodesic
\newcommand{\capac}{\ensuremath{\mathrm{cap}}} %capacity
\newcommand{\W}{\mathcal{W}} % quotient space
\newcommand{\h}{{\mathbb H}^2} %hyperbolic space

\newcommand{\D}{\mathbb{D}} %unit disk

\newcommand{\Pm}{\mathbb{P}}

\newcommand{\Om}{\Omega}
\newcommand{\M}{\mathcal{M}}% Arbitrary sigma algebra

\newcommand{\E}{ \mathbb{E}}

\newcommand{\eqdist}{\overset{d}{=}} %equal in distribution
\newcommand{\pspace}{\mathcal{P}} %space of probability measures
\newcommand{\N}{\mathbb{N}}
\newcommand{\Z}{\mathbb{Z}}

\newcommand{\R}{\mathbb{R}}
\newcommand{\C}{\mathbb{C}}
\newcommand{\B}{\mathcal{B}} % Borel sigma

\newcommand{\V}{\mathcal{V}}

\newcommand{\remark}{\noindent \textbf{Remark. }}

\newcommand{\id}{\mathrm{d}}

\newcommand{\im}{\mathrm{i}}
\newcommand{\e}{\mathrm{e}}

\newcommand{\supp}{\mathrm{supp}\,}

\newcommand{\dist}{\mathrm{dist}}

\newcommand{\imag}{\mathrm{Im} \,}
\newcommand{\re}{\mathrm{Re} \, }

\newtheorem{thm}{Theorem}
\newtheorem{lem}{Lemma}
\newtheorem{prop}{Proposition}

\newtheorem*{thm*}{Theorem}
\theoremstyle{definition}

\numberwithin{thm}{section}
\numberwithin{lem}{section}
\numberwithin{prop}{section}
\numberwithin{cor}{section}

\newcommand{\pvis}{P_{\mbox{vis}}}
\begin{document}

%\listoftodos

\pagenumbering{arabic}
\title{\bf{Visibility in the vacant set of the Brownian interlacements and the Brownian excursion process} \rm}

\author{Olof Elias\footnote{Department of Mathematics, Chalmers University of Technology and Gothenburg University, Sweden. E-mail: olofel@chalmers.se} \and Johan Tykesson\footnote{Department of Mathematics, Chalmers University of Technology and Gothenburg University, Sweden. E-mail: johant@chalmers.se. Research supported by the Knut and Alice Wallenberg foundation.} }

\date{\today}

\maketitle
\thispagestyle{empty}

\begin{abstract}
We consider the Brownian interlacements model in Euclidean space, introduced by A.S. Sznitman in \cite{sznitman2013scaling}. We give estimates for the asymptotics of the visibility in the vacant set. We also consider visibility inside the vacant set of the Brownian excursion process in the unit disc and show that it undergoes a phase transition regarding visibility to infinity as in \cite{benjamini2009visibility}. Additionally, we determine the critical value and that there is no visibility to infinity at the critical intensity.
\end{abstract}
%% header
\pagestyle{fancy}
\setlength{\headheight}{14pt} 
\fancyhf{}
\lhead{Visibility in Brownian interlacements}
\cfoot{\thepage}
%\rhead{Olof Elias}
%\lhead{FKG inequality and visibility in the Brownian interlacements model}
\cfoot{\thepage}
\section{Introduction}
In this paper, we study visibility inside the vacant set of two percolation models; the Brownian interlacements model in ${\mathbb R}^d$ ($d\ge 3$), and the Brownian excursion process in the unit disc. Below, we first informally discuss Brownian interlacements model and our results for that model, and then we move on the Brownian excursions process.

The Brownian interlacements model is defined as a Poisson point process on the space of doubly infinite continuous trajectories modulo time-shift in ${\mathbb R}^d$, $d\ge 3$. The aforementioned trajectories essentially look like the traces of double-sided Brownian motions. It was introduced by A.S Sznitman in \cite{sznitman2013scaling} as a means to study scaling limits of the occupation measure of continuous time random interlacements on the lattice $N^{-1} \Z^d$.  The Brownian interlacements model can be considered to be the continuous counterpart of the random interlacements model, which is defined as a Poisson point process on the space of doubly infinite trajectories in ${\mathbb Z}^d$, $d\ge 3$, and was introduced in \cite{szn2010vacant}. Both models exhibit infinite range dependence of polynomial decay, which often complicates the application of standard arguments. Random interlacements on ${\mathbb Z}^d$ have received quite a lot of attention since their introduction. For example, percolation in the vacant set of the model have been studied in \cite{szn2009vacant3} and  \cite{szn2010vacant}. Connectivity properties of the interlacement set have been studied in \cite{rath2012connectivity}, \cite{proc2011geometry}, \cite{popov2012internal} and \cite{lacoin2013easiest}. For the Brownian interlacements model, percolative and connectivity properties were studied in \cite{li2016percolative}.

We will recall the precise definition of the Brownian interlacement model in Section~\ref{s.preliminaries}, where we will also give the precise formulation of our main results, but first we discuss our results somewhat informally. In the present work, we study \emph{visibility} inside the vacant set of the Brownian interlacements. For $\rho>0$ and $\alpha>0$, the vacant set $\V_{\alpha,\rho}$ is the complement of the random closed set $\bi^\rho_\alpha$, which is the closed $\rho$-neighbourhood of the union of the traces of the trajectories in the underlying Poisson point process in the model. Here $\alpha$ is a multiplicative constant of the intensity measure (see~\eqref{e.interlacemeas}) of the Poisson point process, governing the amount of trajectories that appear in the process. The visibility in a fixed direction in  $\V_{\alpha,\rho}$ from a given point $x\in {\mathbb R}^d$ ($d\ge 3$) is defined as the longest distance you can move from $x$ in the direction, without hitting $\bi^\rho_\alpha$. The probability that the visibility in a fixed direction from $x$ is larger than $r\ge 0$ is denoted by $f(r)=f(r,\alpha,\rho,d)$. The visibility from $x$ is then defined as the longest distance you can move in \emph{some} direction, and the probability that the visibility is larger than $r\ge 0$ is denoted by $\pvis(r)=\pvis(r,\alpha,\rho,d)$. Clearly, $\pvis(r)\ge f(r)$, but it is of interest to more closely study the relationship between the functions $\pvis(r)$ and $f(r)$. Our main result for Brownian interlacements in ${\mathbb R}^d$, Theorem \ref{t.euclidmain}, gives upper and lower bounds of $\pvis(r)$ in terms of $f(r)$. In particular, Theorem~\ref{t.euclidmain} show that the rates of decay (in $r$) for the two functions differ with at most a polynomial factor. It is worth mentioning that even if the Brownian interlacements model in some aspects behaves very differently from more standard continuum percolation models like the Poisson Boolean model, when it comes to visibility the difference does not appear to be too big. The proof of Theorem~\ref{t.euclidmain} uses first and second moment methods and is inspired by the proofs of Lemmas $3.5$ and $3.6$ of \cite{benjamini2009visibility}. The existence of long-range dependence in the model creates some extra complications to overcome. It seems to us that the arguments in the proof of Theorem~\ref{t.euclidmain} are possible to adapt to other percolation models based on Poisson-processes on infinite objects, for example the Poisson cylinder model \cite{tykesson2012percolation}.

We now move on to the Brownian excursion process in the open unit disk $\D=\{z\in {\mathbb C}\,:\,|z|<1\}$. This process is defined as a Poisson point process on the space of Brownian paths that start and end on $\partial \D$, and stay inside $\D$ in between. The intensity measure is given by $\alpha \mu$ where $\mu$ is the Brownian excursion measure (see for example~\cite{lawler2000universality}, \cite{lawler2005conformally}) and $\alpha>0$ is a constant. This process was studied in~\cite{wu2012occupation}, where, among other things, connections to Gaussian free fields were made. The union of the traces of the trajectories in this Poisson point process is a closed random set which we denoted by $\be_{\alpha}$, and the complement is denoted by $\V_{\alpha}$. Again, we consider visibility inside the vacant set. In Theorem~\ref{t.BEmain}, we show that, there is a critical level $\alpha_c=\pi/4$ such that if $\alpha <\alpha_c$, with positive probability there is some $\theta\in [0,2\pi)$ such that the line-segment $[0,\e^{i\theta})$ (which has infinite length in the hyperbolic metric) is contained in $\V_{\alpha}$, while if $\alpha\ge \alpha_c$ the set of such $\theta$ is a.s.\ empty. A similar phase transistion is known to hold for the Poisson Boolean model of continuum percolation and some other models in the hyperbolic plane, see \cite{benjamini2009visibility} and \cite{lyons1996diffusion}. As seen by Theorem~\ref{t.euclidmain}, such a phase transition does not occur for the set $\V_{\alpha,\rho}$ in the Brownian interlacements model in Euclidean space, when $\rho>0$. The proof of Theorem~\ref{t.BEmain} is based on circle covering techniques, using a sharp condition by Shepp \cite{shepp1972covering}, see Theorem~\ref{Shepp}, for when the unit circle is covered by random arcs. To be able to use Shepp's condition, the $\mu$-measure of a certain set of trajectories must be calculated. This is done in the key lemma of the section, Lemma~\ref{l.ameasure}, which we think might be of independent interest. Lemma ~\ref{l.ameasure} has a somewhat surprising consequence, see Equation~\eqref{e.shadowproc}. 

We now give some historical remarks concerning the study of visibility in various models. The problem of visibility was first studied by G.Pólya in \cite{polya1918zahlentheoretisches} where he considered the visibility for a person at the origin and discs of radius $R>0,$ placed on the lattice $\Z^2$. For the Poisson Boolean model of continuum percolation in the Euclidean plane, an explicit expression is known for the probability that the visibility is larger than $r$, see Proposition 2.1 on p.4 in \cite{calka2009asymptotics} (which uses a formula from \cite{holst1982covering}). Visibility in non-Euclidean spaces has been considered by R.Lyons in \cite{lyons1996diffusion}, where he studied the visibility on manifolds with negative curvature, see also Kahanes earlier works \cite{kahane1990coverings} \cite{kahane1991coverings} in the two-dimensional case. In the hyperbolic plane, visibility in so-called well behaved random sets was studied in \cite{benjamini2009visibility} by Benjamini et.\ al.

The rest of the paper is organized as follows. In Section~\ref{s.preliminaries} we give the definitions of Brownian interlacements and Brownian excursions, and give the precise formulations of our results. Section~\ref{s.preleuclid} contains some preliminary results needed for the proof of our main result for Brownian interlacements. In Section~\ref{s.euclidproof} we prove the main result for Brownian interlacements. The final section of the paper, Section~\ref{s.BEproof}, contains the proof of our main result for the Brownian excursion process.

We now introduce some notation. We denote by $1\{A\}$ the indicator function of a set $A$. By $A \Subset X$ we mean that $A$ is a compact subset of a topological space $X$.
Let $a \in [0,\infty]$ and $f,g$ be two functions. If $\limsup_{x\to a} f/g = 0$ we write $f = o(g(x))$ as $x \to a $, and if $\limsup_{x\to a} f/g <\infty$ we write $f = O(g(x))$ as $x\to a$.
We write $f(x) \sim g(x)$ as $x \to a$ to indicate that $\lim_{x\to a} (f(x)/g(x)) = 1$ and $f(x) \lesssim g(x)$  as $x \to a$ to indicate that $f(x) \leq g(x)(1+o(1))$ as $x\to a$. For $x\in {\mathbb R}^d$ and $r>0$, let $B(x,r)=\{y\,:\,|x-y|\le r\}$ and $B(r) = B(0,r)$. For $A \subset \R^d$ define
\[
A^t := \left\{ x \in \R^d : \dist(x,A) \leq t \right\},
\]
to be the closed $t$-neighbourhood of A. For $x,y \in \R^d$ let $[x,y]$ be the (straight) line segment between $x$ and $y$.

Finally, we describe the notation and the convention for constants used in this paper. We will let $c,c',c''$ denote positive finite constants that are allowed to depend on the dimension $d$ and the thickness $\rho$ only, and their values might change from place to place, even on the same line. With numbered constants $c_i$, $i\ge 1$, we denote constants that are defined where they first appear within a proof, and stay the same for the rest of the proof. If a constant depends on another parameter, for example the intensity of the underlying Poisson point process, this is indicated.

%\item $f(x) \gtrsim g(x), \ \Leftrightarrow g(x) \lesssim f(x)$
%\item $ f(x) \asymp g(x),\ $ if there exists $0<c<C$ such that $cg(x) \leq f(x) \leq C g(x)$. Moreover we write this as $f(x) = \Theta(g(x))$.
%\end{itemize}

\section{Preliminaries}\label{s.preliminaries}

\subsection{Brownian interlacements}
We begin with the setup as in \cite{sznitman2013scaling}. Let $C= C(\R;\R^d)$ denote the continuous functions from $\R$ to $\R^d$ and let  $C_+=C(\R_+;\R^d)$ denote the continuous functions from $\R_+$ to $\R^d$. Define
\begin{equation*}
W = \{ x \in C:  \lim_{|t| \to \infty} |x(t)| = \infty \}\mbox{ and }W_+ = \{ x \in C_+:  \lim_{t \to \infty} |x(t)| = \infty \}.
\end{equation*}
On $W$ we let $X_t$, $t\in \R$, denote the canonical process, i.e. $X_t(w) = w(t)$ for $w \in C$, and let $\W$ denote the $\sigma$-algebra generated by the canonical processes. Moreover we let $\theta_x, x \in \R$ denote the shift operators acting on $\R$, that is $\theta_x : \R \to \R, y\mapsto y+x$. We extend this notion to act on $C$ by composition as
 \[
 \theta_x : C \to C, f\mapsto f \circ \theta_x.
 \]
 Similarly, on $W_+$, we define the canonical process $X_t$, $t\ge 0$, the shifts $\theta_h$, $h\ge 0$, and the sigma algebra $\W_+$ generated by the canonical processes.
We define the following random times corresponding to the canonical processes. For $F\subset \R^d$ closed and $w\in W_+$, the {\bf entrance time} is defined as $H_F(w) = \inf\{ t \ge 0 : X_t(w) \in F\}$ and the {\bf hitting time} is defined as $\tilde{H}_F(w) = \inf\{ t> 0: X_t(w) \in F\}$. For $K \Subset \R^d$ the {\bf time of last visit} to $K$ for $w \in W_+$ is defined as $L_K(w) = \sup\left\{ t>0 : X_t(w) \in K \right\}$. The entrance time for $w\in W$ is defined similarly, but $t>0$ is replaced by $t\in \R$.
%\begin{description}
%\item[\bf{Exit time}] $T_U(w) = \inf\{ t\in \R: X_t(w) \notin U \}$ for $U\subset \R^d$ open,
%\item[\bf{Entrance time}] $H_F(w) = \inf\{ t \in \R : X_t(w) \in F\}$,
%\item[\bf{Hitting time}] If $t \in [0,\infty)$ then $\tilde{H}_F(w) = \inf\{ t> 0: X_t(w) \in F\}$, for $F\subset \R^d$ closed,
%\item[\bf{Time of last visit}] $K\subset \R^d$ compact, $L_K(w):= \sup\{t \in T : w(t) \in K\}.$
%\end{description}
On $W$, we introduce the equivalence relation $w \sim w' \Leftrightarrow \exists h \in \R : \theta_h w = w'$ and we denote the quotient space by $W^* = W/\sim$ and let 
\[
\pi :W \to W^*,\ w \mapsto w^*,
\]
denote the canonical projection. Moreover, we let $\W^*$ denote the largest $\sigma$-algebra such that $\pi$ is a measurable function, i.e. $\W^* = \{ \pi^{-1}(A) :A \in \W  \}$. We denote $W_K \subset W$ all trajectories which enter $K$, and $W^*_K$ the associated projection.
We let $P_x$ be the Wiener measure on $C$ with the canonical process starting at $x$, and we denote $ P_x^B (\cdot) = P_x(\cdot | H_B = \infty)$ the probability measure conditioned on the event that the Brownian motion never hits $B$. For a finite measure $\lambda$ on $\R^d$ we define $$P_\lambda=\int P_x \lambda(dx).$$
The transition density for the Brownian motion on $\R^d$ is given by
\begin{equation}
  p(t,x,y) := \frac{1}{(2 \pi t)^{d/2}}\exp\left(-\frac{|x-y|^2}{2t} \right)
\end{equation}
and the Greens function is given by
\[
G(x,y) = G(x-y) := \int_0^\infty p(t,x,y) \id t = c_d / |x-y|^{d-2},
\]
where $c_d$ is some dimension dependent constant, see Theorem 3.33 p.80 in \cite{morters2010brownian}.

Following \cite{sznitman2013scaling} we introduce the following potential theoretic framework. For  $K \Subset \R^d$ let $\pspace(K)$ be the space of probability measures supported on $K$ and introduce the energy functional
\begin{equation}\label{energy}
E_K(\lambda) = \int_{K\times K} G(x,y) \lambda(\id x) \lambda( \id y),\ \lambda \in \pspace(K).
\end{equation}

The Newtonian capacity of $ K \Subset \R^d$ is defined as
\begin{equation}
\capac(K):= \left( \inf_{\lambda \in \pspace(K)} \left\{ E_K(\lambda) \right\} \right)^{-1},
\end{equation}
see for instance \cite{brelot1967lectures}, \cite{portstone} or \cite{morters2010brownian}. 
It is the case that
\begin{equation}\label{e.capprop}
\mbox{ 
the capacity is a strongly sub-additive and monotone set-function.
}
\end{equation}

Let $e_K( \id y)$ be the equilibrium measure, which is the finite measure that is uniquely determined by the last exit formula, see Theorem 8.8 in \cite{morters2010brownian},
\begin{equation}
P_x( X(L_K) \in A) = \int_A G(x,y) e_K( \id y),
\end{equation}
and let $\tilde{e}_K$ be the normalized equilibrium measure. By Theorem 8.27 on p. 240 in \cite{morters2010brownian} we have that $\tilde{e}_K$ is the unique minimzer of \eqref{energy} and
\begin{equation}
\capac(K) = e_K(K).
\end{equation}
Moreover the support satisfies $\supp e_K(\id y) = \partial K$. 

If $B$ is a closed ball, we define the measure $Q_B$ on $W_B^0:=\{w\in W\,:\, H_B(w)=0\}$ as follows:
\begin{equation}\label{e.qdef}
Q_B \left[(X_{-t})_{t \geq 0} \subset A', \ X_0 \in \id y,\ (X_t)_{t \geq 0} \subset A\right] := P^B_y(A') P_y(A) e_B(\id y),
\end{equation}
where $A,A'\in {\mathcal W}_{+}$.
If $K$ is compact, then $Q_K$ is defined as 
$$Q_K=\theta_{H_K}\circ (1\{H_K<\infty\}Q_B),\mbox{ for any closed ball }B\supseteq K.$$
As pointed out  in \cite{sznitman2013scaling} this definition is independent of the choice of $B\supseteq K$ and coincides with~\eqref{e.qdef} when $K$ is a closed ball. We point out that Equation $2.21$ of \cite{sznitman2013scaling} says that

\begin{equation}\label{e.forwardtimes}
Q_K[(X_t)_{t\ge 0}\in \cdot]=P_{e_K}(\cdot).
\end{equation}

From \cite{sznitman2013scaling} we have the following theorem, which is Theorem 2.2 on p.564.
\begin{thm}
There exists a unique $\sigma$-finite measure $\nu$ on $(W^*, \W^*)$ such that for all $K$ compact,
\begin{equation}\label{e.interlacemeas}
\nu(\cdot \cap W^*_K) = \pi \circ Q_K(\cdot).
\end{equation}
\end{thm}
By Equations $2.7$ on p.564 and $2.21$ on p.568 in \cite{sznitman2013scaling} it follows that for $K\subset {\mathbb R}^d$ compact
$$\nu(W_K^*)={\rm cap}(K).$$
Now we introduce the space of point measures or configurations, where $\delta$ is the usual Dirac measure:
\small\begin{equation}
\Om = \left\{\omega = \sum_{i \geq 0 } \delta(w^*_i, \alpha_i) : (w^*_i,\alpha_i) \in W^*\times[0,\infty),\ \omega(W^*_K\times [0,\alpha])<\infty, \forall K \Subset \R^d, \alpha \geq 0 \right\},
\end{equation}\normalsize
and we endow $\Om$ with the $\sigma$-algebra $\M$ generated by the evaluation maps \[\omega \mapsto \omega(B), B\in \W^*\otimes \B(\R_+).\]
Furthermore, we let $\Pm$ denote the law of the Poisson point process of $W^*\times  \R_+$ with intensity measure $\nu \otimes \id \alpha$.
The Brownian interlacement is then defined as the random closed set 
\begin{equation}
\bi^\rho_\alpha (\omega) := \bigcup_{\alpha_i \le \alpha } \bigcup_{s \in \R} B(w_i(s),\rho), 
\end{equation}

where $\omega = \sum_{i \geq 0 } \delta(w^*_i, \alpha_i) \in \Omega$ and $\pi(w_i)=w_i^*.$
We then let $\V_{\alpha,\rho} = \R^d\setminus \bi^\rho_\alpha$ denote the vacant set.

The law of $\bi^\rho_\alpha$ is characterized as follows. Let $\Sigma$ denote the family of all closed sets of $\R^d$ and let $\mathbb{F}:= \sigma\left(F \in \Sigma: F\cap K = \emptyset, K\ \text{compact} \right)$. The law of the interlacement set, $Q_\alpha^\rho$, is a probability measure on $(\Sigma,\mathbb{F})$ given by the following identity:
\begin{equation}
Q_\alpha^\rho \left(\{F \in \Sigma: F \cap K = \emptyset\}\right)=\Pm\left(\bi^\rho_\alpha \cap K = \emptyset\right) = \e^{-\alpha \capac(K^\rho )}.
\end{equation}
For convenience, we also introduce the following notation. For $\alpha>0$ and $\omega=\sum_{i\ge 1}\delta_{(w_i,\alpha_i)}\in \Omega$, we write  

\begin{equation}
\omega_{\alpha}:=\sum_{i\ge 1}\delta_{(w_i,\alpha_i)}1\{\alpha_i\le \alpha\}.
\end{equation}

Observe that under ${\mathbb P}$, $\omega_{\alpha}$ is a Poisson point process on $W^*$ with intensity measure $\alpha \nu$. Note that, by Remark 2.3 (2) and Proposition 2.4 in \cite{sznitman2013scaling} both $\nu$ and $\Pm$ are invariant under translations as well as linear isometries.

\remark
To get a better intuition of how this model works it might be good to think of the local structure of the random set $\bi^{\rho}_{\alpha}$. This can be done in the following way, which uses \eqref{e.forwardtimes}. Let $K\subset \R^d$ be a compact set. Let $N_K\sim\text{Poisson}(\alpha \capac (K))$. Conditioned on $N_K$, let $(y_i)_{i=1}^{N_K}$ be i.i.d.\ with distribution $\tilde{e}_K$.  Conditioned on $N_K$ and $(y_i)_{i=1}^{N_K}$ let $((B_i(t))_{t\ge 0})_{i=1}^{N_K}$ be a collection of independent Brownian motions  in ${\mathbb R}^d$ with $B_i(0)=y_i$ for $i=1,...,N_K$. We have the following distributional equality:
\begin{equation}\label{e.local}
K \cap \bi^\rho_\alpha \eqdist \left(\bigcup_{i=1}^{N_{K}}  [B_i]^\rho  \right) \cap K,
\end{equation}
where $[B_i]$ stands for the trace of $B_i$.
%To clarify, let $K = B(0,1)$, which implies $e_K = \sigma_d$, where $\sigma_d$ is the uniform measure on the sphere, and $N_{e_K} \sim \text{Poisson}(\alpha \omega_d (1+\rho)^{d-2})$.
%
%Thus, in the case of $K$ being a sphere the Brownian interlacements has the same distribution as the union of $N_K$ Brownian motions started uniformly at random on the sphere. 

\subsection{Results for the Brownian interlacements model in Euclidean space}
The following theorem is our main result concerning visibility inside the vacant set of Brownian interlacements in ${\mathbb R}^d$.

\begin{thm}\label{t.euclidmain}
There exist constants $0<c<c'<\infty$ depending only on $d$, $\rho$ and $\alpha$ such that
\begin{align}
\pvis(r)& \lesssim  c'\, r^{2(d-1)} f(r),\ d\ge 3, \label{e.bounds4} \\
  \pvis (r)&\gtrsim c\,r^{d-1} f(r) ,\ d\ge 4, \label{e.bounds5}
\end{align}
as $r \to \infty$.
\end{thm}

We believe that the lower bound in~\eqref{e.bounds5} is closer to the true asymptotic behaviour of $\pvis(r)$ as $r\to \infty$ than the upper bound in~\eqref{e.bounds4}. Indeed, if for $r>0$ we let $Z_r$ denote the set of points $x\in\partial B(0,r)$ such that $[0,x]\subset \V_{\alpha,\rho}$, then the expected value of $|Z_r|$ is proportional to $r^{d-1} f(r)$.  We also observe that a consequence of Theorem~\ref{t.euclidmain} we obtain that $\pvis (r)\to 0$ as $r\to \infty$. However, this fact can be obtained in simpler ways than Theorem~\ref{t.euclidmain}.

\subsection{Brownian excursions in the unit disc}
The Brownian excursion measure on a domain $S$ in $\C$ is a $\sigma$-finite measure on Brownian paths which is supported on the set of continuous paths, $w =(w(t))_{0 \leq t \leq T_w}$, that start and end on the boundary $\partial S$ such that $w(t) \in S, \forall t \in (0,T_w)$. Its definition is found in for example \cite{lawler2000universality}, \cite{virag2003beads}, see also \cite{lawler2005conformally}, \cite{lawler2004soup} for useful reviews. We now recall the definition and properties of the Brownian excursion measure in the case when $S$ is the open unit disc ${\mathbb D}=\{z\in {\mathbb C}\,:\,|z|<1\}$.

 Let 
\[
W_\D := \left\{ w \in C([0,T_w],\bar{\D}) : w(0),w(T_w) \in \partial \D,\ w(t) \in \D, \forall t \in (0,T_w)  \right\}
\]
and let $X_t(w) = w(t)$ be the canonical process on $W_\D$. Let $\W_\D$ be the sigma-algebra generated by the canonical processes. Moreover, for $K \subset \D$ we let $W_{K,\D}$ be the set of trajectories in $W_\D$ that hit $K$. Let 
\small\begin{equation}
\Om_\D = \left\{\omega = \sum_{i \geq 0 } \delta(w_i, \alpha_i) : (w_i,\alpha_i) \in W_\D\times[0,\infty),\ \omega(W_{K,\D}\times [0,\alpha])<\infty, \forall K \Subset \D, \alpha \geq0 \right\}.
\end{equation}\normalsize
We endow $\Om_\D$ with the $\sigma$-algebra $\M_\D$ generated by the evaluation maps \[\omega \mapsto \omega(B), B\in \W_{\D}\otimes \B(\R_+).\]

For a probability measure $\sigma$ on $\D$, denote by $P_{\sigma}$ the law of Brownian motion with starting point chosen at random according to $\sigma$, stopped upon hitting $\partial \D$. (Note that $P_\sigma$ has a different meaning if it occurs in a section concerning Brownian interlacements.) For $r>0$, let $\sigma_r$ be the uniform probability measure on $\partial B(0,r)\subset {\mathbb R}^2$. The Brownian excursion measure on $\D$ is defined as the limit
\begin{equation}
\mu  = \lim_{\epsilon \to 0} \frac{2 \pi}{\epsilon} P_{\sigma_{1-\epsilon}}. 
\end{equation} 
See for example Chapter $5$ in \cite{lawler2005conformally}  for details. The measure $\mu$ is a sigma-finite measure on $W_\D$ with infinite mass.

As in~\cite{wu2012occupation} we can then define the Brownian excursion process as a Poisson point process on $W_\D\times  \R_+$ with intensity measure $\mu \otimes \id \alpha$ and we let $\Pm_\D$ denote the probability measure corresponding to this process.

For $\alpha>0$, the Brownian excursion set at level $\alpha$ is then defined as
\begin{equation}
\be_\alpha (\omega) := \bigcup_{\alpha_i \le \alpha } \bigcup_{s \ge 0} w_i(s),\ \omega = \sum_{i \geq 0 } \delta(w_i, \alpha_i)\in \Omega_{\D},
\end{equation}
and we let $\V_\alpha = \D\setminus \be_\alpha$ denote the vacant set. 

Proposition 5.8 in \cite{lawler2005conformally} says that $\mu$, and consequently ${\mathbb P}_{\D}$, are invariant under conformal automorphisms of $\D$. The conformal automorphisms of $\D$ are given by 
\begin{equation}
T_{\lambda, a} = \lambda \frac{z-a}{\bar{a} z - 1},\ |\lambda|=1,\ |a|<1.
\end{equation}
On  $\D$ we consider the hyperbolic metric $\rho$ given by

$$\rho(u,v)=2\tanh^{-1}\left| \frac{u-v}{1-\bar{u}v} \right|\mbox{ for }u,v\in \D.$$

We refer to $\D$ equipped with $\rho$ as the Poincar\'e disc model of $2$-dimensional hyperbolic space ${\mathbb H}^2$. The metric $\rho$ is invariant under $(T_{\lambda, a})_{|\lambda|=1,\ |a|<1}.$

%For $\theta\in [0,2\pi)$, we denote by $L_{\infty}(\theta)$ the line segment $[0,\e^{i\theta})$. 

The Brownian excursion process can in some sense be thought of as the $\h$ analogue of the Brownian interlacements process due to the following reasons.  As already mentioned that the law of the Brownian excursion process is invariant under the conformal automorphisms of $\D$, which are isometries of $\h$. Moreover, Brownian motion in $\h$ started at $x\in \D$ can be seen as a time-changed Brownian motion started at $x$ stopped upon hitting $\partial \D$, see Example 3.3.3 on p.84 in \cite{hsu2002stochastic}. In addition, we can easily calculate the $\mu$-measure of trajectories that hit a ball as follows. First observe that for $r<1$ 
\begin{align}\label{e.meascapac}
\mu(\{\gamma\,:\,\gamma\cap B(0,r)\neq \emptyset\})& =\lim_{\epsilon\to 0}2 \pi \epsilon^{-1} P_{\sigma_{1-\epsilon}}(H_{B(0,r)}<\infty)\nonumber\\ & =\lim_{\epsilon\to 0}\frac{2 \pi \log(1-\epsilon)}{\epsilon \log(r)}=-\frac{2 \pi}{\log(r)},
\end{align}
where we used Theorem 3.18 of~\cite{morters2010brownian} in the penultimate equality. For $r_h\ge 0$ let $B_{{\mathbb H}^2}(x,r_h)=\{y\in \D\,:\,\rho(x,y)\le r_h\}$ be the closed hyperbolic ball centered at $x$ with hyperbolic radius $r_h$. Then $B_{{\mathbb H}^2}(0,r_h)=B(0,(e^{r_h}-1)/(e^{r_h}+1))$ so that 
\begin{align*}
\mu(\{\gamma\,:\,\gamma\cap B_{{\mathbb H}^2}(0,r_h)\neq \emptyset\})=-\frac{2 \pi}{\log(\frac{e^{r_h}-1}{e^{r_h}+1})}=\frac{2 \pi}{\log(\coth(r_h/2))}.
\end{align*}
The last expression can be recognized as the hyperbolic capacity (see \cite{grigor1999analytic} for definition) of a hyperbolic ball of radius $r_h$, since according to Equation 4.23 in  \cite{grigor1999analytic}

\begin{equation}\label{ball_capac}
\capac_{\h}(B_{\h}(0,r_h))=\left( \int_{r_h}^{\infty} \frac{1}{S(t)} \id t\right)^{-1},
\end{equation}
where $S(r_h)=2\pi \sinh(r_h)$ is the circumference of a ball of radius $r_h$ in the hyperbolic metric. The integral equals
\[
\int_{r_h}^{\infty} \frac{1}{2\pi \sinh(t)} \id t = \frac{1}{2\pi }\left[ \log(\tanh(t/2)) \right]_{r_h}^\infty =\frac{\log(\coth(r_h/2))}{2\pi },
\]
which yields the expression
\[
\capac_{\h}(B_{\h}(0,r_h)) = \frac{2 \pi}{ \log[\coth(r_h/2)]},
\]
which coincides with~\eqref{e.meascapac}.

We now define the event of interest in this section. Let

\begin{equation}
V_\infty ^{\alpha}= \left\{  \{\theta\in [0,2\pi)\,:\, [0,\e^{i\theta})\subset \V_\alpha\}\neq \emptyset \right\}.
\end{equation}

If $V_{\infty}^{\alpha}$ occurs, we say that we have visibility to infinity in the vacant set (since $[0,\e^{i\theta})$ has infinite length in the hyperbolic metric). As remarked above, such a phenomena cannot occur for the Brownian interlacements model on ${\mathbb R}^d$ $(d\ge 3)$.

\subsection{Results for the Brownian excursions process}
Our main result (Theorem~\ref{t.BEmain}) for the Brownian excursion process is that we have a phase transition for visibility to infinity in the vacant set. We also determine the critical level for this transition and what happens at the critical level. 
\begin{thm}\label{t.BEmain}
It is the case that
\begin{equation}
\left. 
\begin{matrix}
\Pm_\D(V_\infty^{\alpha}) > 0, & \alpha < \pi /4, \\
\Pm_\D(V_\infty^{\alpha}) =0, & \alpha\geq \pi /4.
\end{matrix}
\right.
\end{equation}
\end{thm}

\remark A similar phase-transition for visibility to infinity was proven to hold for so called well-behaved random sets in the hyperbolic plane in \cite{benjamini2009visibility}. One example of a well-behaved random set is the vacant set of the Poisson-Boolean model of continuum percolation with balls of deterministic radii. In this model, balls of some fixed radius are centered around the points of a homogeneous Poisson point process in $\h$, and the vacant set is the complement of the union of those balls. In this case, a phase-transition for visibility was known to hold earlier, see \cite{lyons1996diffusion}.

\remark It is easy to see that 

\begin{equation}
{\mathbb P}_{{\mathbb D}}([0,\e^{i\theta})\subset \V_\alpha)=0\mbox{ for every }\theta\in [0,2\pi)\mbox{ and every }\alpha>0.
\end{equation}

Hence, the set $\{\theta\in [0,2\pi)\,:\, [0,\e^{i\theta})\subset \V_\alpha\}$ has Lebesgue measure $0$ a.s.\ when $\alpha>0$. It could be of interest to determine the Hausdorff dimension of $\{\theta\in [0,2\pi)\,:\, [0,\e^{i\theta})\subset \V_\alpha\}$ on the event that this set is non-empty. This was for example done for well-behaved random sets in the hyperbolic plane in \cite{thaele2014hausdorff}.
\section{Preliminary results for the Euclidean case}\label{s.preleuclid}
%Now we turn to estimating the probability of visibility, that is the probability of the existence of a straight line, from $0$ to $\partial B(r),\ r >0$. 
In this section we collect some preliminary results needed for the proof of Theorem~\ref{t.euclidmain}.
The parameters $\alpha>0$ and $\rho>0$ will be kept fixed, so for brevity we write $\V$ and $\bi$ for $\V_{\alpha,\rho}$ and $\bi^{\rho}_\alpha$ respectively. We now introduce some additional notation.
For $A,B \Subset \R^d$ define the event
\begin{equation}\label{visAB}
A \geo B := \{ \exists\, x\in A,\,y\in B\,:\,[x,y]\subset \V \}.
\end{equation}

Then
\begin{align}
& \pvis (r) = \Pm \left(0 \geo \partial B(r)\ \right), \\
&f(r) = \Pm \left(  0 \geo xr \right), x \in S^{d-1},
\end{align}
where $S^{d-1}= \partial B(1)$. For $L,\rho>0$ let 
\begin{equation}
[0,L]_{\rho} := \left\{ x = (x_1,x') \in \R^d : x_1 \in[0,L],\ | x'| \leq \rho \right\}.
\end{equation}
For $x,y\in \R^d$ let $[x,y]_{\rho}=R_{x,y}([0,|x-y|]_{\rho})$ where $R_{x,y}$ is an isometry on $\R^d$ mapping $0$ to $x$ and $(|x-y|,0,...,0)$ to $y$. In other words, $[x,y]_{\rho}$ is the finite cylinder with base radius $\rho$ and with central axis running between $x$ and $y$. Using estimates of the capacity of $[0,L]_{1}$ from \cite{portstone} we easily obtain estimates of the capacity of $[0,L]_{\rho}$ for general $\rho$ as follows.
\begin{lem}\label{l.capacest}
For every $L_0 \in (0,\infty)$ and $\rho_0 \in (0,\infty)$ there are constants $c,c' \in (0,\infty)$ (depending on $L_0,\rho_0$ and $d$) such that for $ L\geq L_0,\ \rho \leq \rho_0$,
\begin{align*}
& c \rho^{d-3} L \leq \capac( [0,L]_{\rho}) \leq c' \rho^{d-3} L,\ d>3,\ \\ 
& c L / ( \log( L/ \rho) )\leq \capac( [0,L]_{\rho}) \leq c' L /( \log( L/ \rho) ),\ d=3.
\end{align*}
\end{lem}
\begin{proof}
Fix $L_0,\rho_0\in (0,\infty)$ and consider $L\ge L_0$ and $\rho\le \rho_0$. Note that $[0,L]_{\rho} = \rho [0,L/\rho]_{1}$. Hence by the homogeneity property of the capacity, see Proposition 3.4 p.67 in \cite{portstone}, we have
\[
\capac ( [0,L]_{\rho}) = \rho^{d-2} \capac([0,L/\rho]_{1}).
\]
We then utilize the following bounds, see Proposition 1.12 p.60 and Proposition 3.4 p.67 in \cite{portstone}: For each $L_0'\in (0,\infty)$ there are constants $c,c'$ such that
\begin{align*}
& c L\le \capac([0,L]_{1}) \le c'L,\ d >3,\\
& c L/\log(L) \le \capac([0,L]_{1}) \le c'L/\log(L),\ d =3,
\end{align*}
for $L\ge L_0'$. The results follows, since $L/\rho\ge L_0/\rho_0$.
\end{proof}
Observe that by invariance, Proposition 3.4 p.67 in \cite{portstone},
\[
\capac([x,y]_\rho) = \capac( [0,|x-y|]_\rho).
\]
Next, we discuss the probability that a given line segment of length $r$ is contained in $\V$, that is
$f(r)$. Note that for $x,y\in \R^d$,
\[
\{x \geo y \}   =\left\{\omega \in \Om : \omega_\alpha \left(W^*_{ [x,y]^\rho} \right)=0\right\}.
\]  
Since under $\Pm$, $\omega$ is a Poisson point process with intensity measure $\nu \otimes \id \alpha $ we get that
\begin{equation}\label{visibility}
 f(|x-y|) = \e^{-\alpha \capac([x,y ]^\rho) }.
\end{equation}
 
Since $[x,y]^\rho$ is the union of the cylinder $[x,y]_\rho$ and two half-spheres of radius $\rho$, it follows using \eqref{e.capprop} that

\begin{equation}
c(\alpha)\e^{-\alpha \capac([x,y ]_\rho)} \le f(|x-y|)\le \e^{-\alpha \capac([x,y ]_\rho) }.
\end{equation}

% Now, note that 
%\[ 
%[0,rx' ]_\rho \subset [0,rx' ]^\rho \subset [-\rho x,(r+\rho)x' ]_\rho,
%\]
%which implies 
%\begin{equation}\label{prelest}
%\e^{-\alpha \capac([-\rho x,(r+\rho)x' ]_\rho) }  \leq f(r) \leq \e^{-\alpha \capac([0,rx' ]_\rho) }.
%\end{equation}

%\begin{lem}\label{visibility_line} For $r>0,$ there exists constants $c(d),c'(\rho,d),C(d)>0$ such that
%\begin{align}
%c' \e^{  -c \alpha r\rho^{d-3} }
%\leq
%&f(r)
%\leq
% \e^{  -C \alpha r\rho^{d-3} },\ d\geq 4  \\
%c'  \e^{-cr/\rho\log(r/\rho) }
%\leq
%&f(r)
%\leq
% \e^{-Cr/\rho\log(r/\rho) },\ d=3.
%\end{align}
%\end{lem}
%

%A similar result holds for $f_{cyl}$ since $\mu(\Line_{C(L,\rho)}) = | \partial C([0,rx],\rho)| = c(d,\rho) r$.
%We end this section with two lemmas that will be useful when proving the lower bound
%\begin{lem}\label{est_meas}
%Let $d \geq 4$, $L$ be a line and $c(L,\rho)=c(L)$ its associated cylinder of radius $\rho$. Then there is a positive constant $c(d)<\infty$ such that for all $K \Subset \R^d$ satisfying $\dist(K,c(L,\rho))> c$ it holds that
%\begin{equation}
%\nu(W^*_K\setminus W^*_{c(L)} ) \geq \frac{1}{2} \nu(W^*_K).
%\end{equation}
%\end{lem}
%\begin{proof}
%Since
%\[
%\nu(W^*_K) = \nu(W^*_K \setminus W^*_{c(L)}) + \nu(W^*_K \cap W^*_{c(L)}),
%\]

The next lemma will be used in the proof of~\eqref{e.bounds5}. 

\begin{lem}\label{est_meas}
Let $d \geq 4$ and $L$ be a bi-infinite line. Let $L_r$ be a line segment of length $r\ge 1$. There are constants $c(d,\rho),\,c'(d,\rho)$ such that
\begin{equation}
\nu(W^*_{L_r^{\rho}}\setminus W^*_{L^{\rho} }) \geq (1-c\, \dist(L_r,L)^{-(d-3)} )\nu(W^*_{L_r^{\rho}}).
\end{equation}
whenever $\dist(L_r,L)\ge c'$.
\end{lem}
\begin{proof}
For simplicity we assume through the proof that $r\ge 1$ is an integer and that one of the endpoints of $L_r$ minimizes the distance between $L$ and $L_r$. The modification of the proof to the case of general $r\ge 1$ and general orientations of the line and the line-segment is straightforward. We write
\begin{equation}\label{e.setsplit}
\nu(W^*_{L_r^{\rho}}) = \nu(W^*_{L_r^{\rho}} \setminus W^*_{L^{\rho}}) + \nu(W^*_{L_r^{\rho}} \cap W^*_{L^{\rho}}),
\end{equation}
and focus on finding a useful upper bound of the second term of the right hand side.

We now write $L = \left(  \gamma_1(t) \right)_{t \in \R}$, where $\gamma_1$ is parametrized to be unit speed and such that $\dist(L_r,\gamma_1(0)) = \dist(L_r,L).$ Similarly, we write $L_r=\left(\gamma_2(t)\right)_{0\le t\le r}$ where $\gamma_2$ has unit speed and $\dist(\gamma_2(0),L)=\dist(L_r,L)$. For $i\in \Z$ and $0\le j\le r-1$ let $y_i = \gamma_1(i)$  and let $z_j=\gamma_2(j)$. Choose $s = s(\rho)<\infty$ such that 
\[
L^{\rho} \subset \bigcup_{i \in \Z} B(y_i,s)\mbox{ and }L^{\rho}_r \subset \bigcup_{i =0}^{r-1} B(z_i,s).
\]
We now have that 
\begin{align*}
\nu(W^*_{L_r^{\rho}} \cap W^*_{L^{\rho}})&\le \sum_{i\in \Z}\sum_{j=0}^{r-1} \nu(W^*_{B(z_j,s)}\cap W^*_{B(y_i,s)})\\ & \le\sum_{i\in \Z}\sum_{j=0}^{r-1} \frac{c}{|z_j-y_i|^{(d-2)}}\le  c\,r  \sum_{i\in \Z} \frac{1}{|z_0-y_i|^{d-2}}\\ & \le c\,r \sum_{i\in \Z} \frac{1}{\left( \dist(L,L_r)^2 +i^2\right)^{\frac{d-2}{2}}}\le c_1\, r \,\dist(L,L_r)^{-(d-3)},
\end{align*}
where the second inequality follows from Lemma 2.1 on p.14 in \cite{li2016percolative}.
Combining this with the fact from Lemma~\ref{l.capacest} that $\nu(W^*_{L_r^\rho})\ge c_2 r$ whenever $r\geq 1 $, we get that
$$ \nu(W^*_{L_r^{\rho}} \cap W^*_{L^{\rho}})\le \frac{c_1}{c_2} \nu(W^*_{L_r^\rho})   \dist(L,L_r)^{-(d-3)},$$
which together with~\eqref{e.setsplit} gives the result.

\end{proof}

\remark
Observe the the Lemma above implies that for every $r>1$, and every line $L$ and line-segment $L_r$ of length $r$ satisfying $\dist(L,L_r)>c,$ we have
\[
\nu(W^*_{L_r^\rho} \setminus W^*_{L^\rho}) \geq \frac{1}{2} \nu(W^*_{L_r^\rho}).
\]
It is easy to generalize the statement to hold for every $r>0$.
 %%%% THE LOWER BOUND %%%%
\section{Proof of Theorem~\ref{t.euclidmain}}\label{s.euclidproof}

We split the proof of Theorem~\ref{t.euclidmain} into the proofs of two propositions, Proposition~\ref{lowerbound} which is the lower bound~\eqref{e.bounds5} and Proposition~\ref{upperbound} which is the upper bound~\eqref{e.bounds4}.

\subsection{The lower bound}

To get a lower bound we will utilize the second moment method. More precisely we shall modify the arguments from the proof of Lemma $3.6$ on p.$332$ in \cite{benjamini2009visibility}. Let
 $\sigma(\id x)$ denote the surface measure of $S^{d-1}$, and for $r>0$ define
\begin{align}
&  Y_r := \left\{ x \in S^{d-1} :  [0 , r x] \subset \V \right\},\\
&  y_r:= |Y_r| = \int_{S^{d-1}} \mathbbm{1}_{Y_r}(x) \sigma(\id x).
\end{align}
The expectation and the second moment of $y_r$ are computed using Fubini's theorem:
\begin{align}
&\E (y_r) = |S^{d-1}|f(r)\label{e.moment1} \\
&\E (y_r^2) = \int_{(S^{d-1})^2} \Pm(x,x' \in Y_r) \sigma(\id x) \sigma(\id x'),\label{e.moments}
\end{align}
where $f(r)$ is given by \eqref{visibility} above. The crucial part of the proof of the lower bound in~\eqref{e.bounds4} is estimating~\eqref{e.moments} from above.

\begin{prop}\label{lowerbound}
Let $d \geq 4$. There exist constants $c(\alpha) , c' $ such that 
\begin{equation}
\pvis(r) \geq c\, r^{d-1} f(r) \mbox{ for all }r\ge c'.
\end{equation}
\end{prop}
\begin{proof}
For $x\in S^{d-1}$ let $L_{\infty}(x)$ be the infinite half-line starting in $0$ and passing through $x$.  For $x,x' \in S^{d-1}$ define $\theta=\theta(x,x') := \arccos \left(  \langle x,x'  \rangle \right)$ to be the angle between the two half-lines $L_{\infty}(x)$ and $L_{\infty}(x')$. From Lemma~\ref{est_meas} and the remark thereafter we know that there is a constant $c_1 $ such that for every $r>0$, and every line $L$ and line-segment $L_r$ of length $r$ satisfying $\dist(L,L_r)\ge c_1$, we have 

\begin{equation}\label{e.halfmeas}
\nu(W^*_{L_r^\rho}\setminus W^*_{L^{\rho}})\ge \frac{1}{2} \nu(W^*_{L_r^{\rho}}).
\end{equation} 

Now define $g(\theta)\in(0,\infty)$ by the equation
\begin{equation}
\dist( L_{\infty}(x), L_{\infty}(x') \setminus [0,g(\theta) x']) =c_1.
\end{equation}
Elementary trigonometry shows that if $\theta \in [0, \pi/2]$ we have
\[
g( \theta) =\frac{c_1}{\sin(\theta)},
\] 
and for $\theta \in [\pi/2 , \pi]$ it is easy to see that we have $g(\theta) \leq c$. Now, for $x,x' \in S^{d-1}$,
\begin{align*}
&\Pm(x,x' \in Y_r)  \leq \Pm\left( [0,rx] \subset \V , [0,rx'] \setminus [0,g(\theta) x'] \subset \V \right) \\
& = \Pm \left(  \omega_{\alpha} \left(W^*_{[0,rx]^\rho}\right) = 0,  \omega_{\alpha} \left(W^*_{([0,rx'] \setminus [0,g(\theta) x'])^\rho}\right) = 0 \right) \\
& \leq  \Pm \left(  \omega_{\alpha} \left(W^*_{[0,rx]^\rho}\right) = 0,  \omega_{\alpha} \left(W^*_{([0,rx'] \setminus [0,g(\theta) x'])^\rho} \setminus W^*_{[0,rx]^\rho}\right) = 0 \right) \\
&\overset{indep.}{=}\Pm \left(  \omega_{\alpha} \left(W^*_{[0,rx]^\rho}\right) = 0 \right) \Pm \left(  \omega_{\alpha} \left(W^*_{([0,rx'] \setminus [0,g(\theta) x'])^\rho} \setminus W^*_{[0,rx]^\rho}\right) = 0 \right) \\
&= f(r) \exp\left\{- \alpha \nu \left(W^*_{([0,rx'] \setminus [0,g(\theta) x'])^\rho} \setminus W^*_{[0,rx]^\rho}\right)\right\} \\
&\stackrel{~\eqref{e.halfmeas}}{\leq} f(r) \exp \left\{-\frac{\alpha}{2} \nu\left( W^*_{(0,((r-g(\theta))\vee 0)x ]^\rho}\right)  \right\} \le f(r) \e^{-(c_2(\alpha) (r-g(\theta)) \vee 0)}c(\alpha),
\end{align*}
where the last inequality follows from Lemma \ref{l.capacest}. Hence, in order to get an upper bound of~\eqref{e.moments} we want to get an upper bound of
\begin{equation}\label{e.idef}
I=\int_{(S^{d-1})^2} \exp\{-c_2((r- g(\theta)) \vee 0)\} \sigma ( \id x) \sigma( \id x').
\end{equation}
In spherical coordinates $\theta,\theta_1,...,\theta_{d-2}$,  we get, with $A(\theta_1,...,\theta_{d-2})=\{(\theta_1,...\theta_{d-2})\,:\,0\le \theta_i<2\pi\mbox{ for all }i\}$, 
\begin{align*}
I&=\int_{\theta= 0}^{\pi/2} \int_{A(\theta_1,...,\theta_{d-2})} \exp\left\{-c_2((r- \frac{c_1 }{\sin(\theta)}) \vee 0 )\right\}  \sin^{d-2}(\theta) \sin^{d-3}( \theta_1) \cdot  \cdot \cdot \sin( \theta_{d-3}) \id \theta \id \theta_1 \cdot \cdot \cdot\id \theta _{d-2} \\
&+\int_{\theta=\pi/2}^{\pi} \int_{A(\theta_1,...,\theta_{d-2})} \exp\left\{-c_2((r-c ) \vee 0 )\right\}  \sin^{d-2}(\theta) \sin^{d-3}( \theta_1) \cdot \cdot \cdot \sin( \theta_{d-3}) \id \theta \id \theta_1 \cdot \cdot \cdot \id \theta _{d-2}\\
&= I_1+I_2.
\end{align*}
We now find an upper bound on the integral $I_1$. We get that
\begin{align}\label{e.twointegrals}
I_1 &\leq c_3 \int_0^{\pi/2} \exp\left\{-c_2((r-\frac{c_1}{\sin(\theta)}) \vee 0)\right\} \sin^{d-2}(\theta) \id \theta\nonumber \\
& = c_3\left(  \int_0^{\arcsin{c_1 /r}} \sin^{d-2} ( \theta) \id \theta + \int_{\arcsin{c_1/r}}^{\pi/2}\e^{-c_2(r-\frac{c_1}{\sin(\theta)}) }  \sin^{d-2} ( \theta) \id \theta \right).
\end{align}
For the first of the two integrals above we get
\begin{align}\label{e.intest1}
&\int_0^{\arcsin{c_1/r}} \sin^{d-2} ( \theta) \id \theta \le c \int_0^{c_1  /r}  \theta^{d-2} \id \theta = c' r^{-(d-1)}.
\end{align}
For the second integral in~\eqref{e.twointegrals} we get (using that $1/\sin(\theta)-1/\theta$ can be extended to a uniformly continuous function on $[0,\pi/2]$)
\begin{align}
&\int_{\arcsin{c_1/r}}^{\pi/2}\e^{-c_2(r-\frac{c_1}{\sin(\theta)}) }  \sin^{d-2} ( \theta) \id \theta \le c\, \e^{-c_2r} \int_{c_1/r}^{ \pi/2} \e^{c_1c_2/ \theta} \theta^{d-2} \id \theta=\nonumber \\
&=c\,\e^{-c_2r} \int_{ 2/\pi}^{r/c_1} \e^{c_1c_2 t} t^{-d} \id t = c\, \e^{-c_2 r} \int_{2 c_1 c_2/ \pi}^{c_2r} \e^y y^{-d} \id y  \nonumber \\ & =  c\, \e^{-c_2 r} \int_{2 c_1 c_2/ \pi}^{c_2 r/2} \e^y y^{-d} \id y +  c\, \e^{-c_2 r} \int_{c_2 r/2 }^{c_2 r} \e^y y^{-d} \id y \nonumber \\ & \le c\, \e^{-c_2 r/2} \int_{2 c_1 c_2 / \pi}^{c_2 r/2} y^{-d} \id y+c\,\int_{c_2 r/2}^{c_2 r} y^{-d} \id y \le c\, r^{-(d-1)}\label{e.intest2}.
\end{align}
Moreover, it is easy to see that
\begin{equation}\label{e.intest3}
I_2 = O( \e^{-c r}).
\end{equation}

Putting equations~\eqref{e.moments},~\eqref{e.idef},~\eqref{e.twointegrals},~\eqref{e.intest1},~\eqref{e.intest2} and ~\eqref{e.intest3} together, we obtain that for all $r$ large enough,
\begin{equation}\label{e.secondmomentupperestimate}
\E [y_r^2 ] \leq c f(r) r^{-(d-1)}.
\end{equation}
From~\eqref{e.moment1},~\eqref{e.secondmomentupperestimate} and the second moment method we get that for all $r$ large enough
\begin{equation*}
\pvis(r)\ge \frac{{\mathbb E}(y_r)^2}{{\mathbb E}(y_r^2)} \ge\,c r^{d-1}\, f(r) ,
\end{equation*}
finishing the proof of the proposition.
\end{proof}

%\remark
%For $d=3$ the best bound available to us the trivial bound $f(r) \leq P_\V(r)$, since the methods available to us does not work well in $\R^3$. For instance, applying the method in \cite{benjamini2009visibility} we can obtain a lower bound of 
%\[
%P_\V (r) \geq c \frac{\log^3(r)}{r} f(r),
%\]
%which is a worse bound than the trivial. The same results holds for the Poisson cylinder model, with the only difference being that it holds for all $d \geq 3$. The argument is completely analogous, with the only difference being that one replaces $W^*$ with $\Line$, and $\nu$ with $\mu$.
 \subsection{The upper bound}
The next proposition is~\eqref{e.bounds4} in Theorem~\ref{t.euclidmain}.

\begin{prop}\label{upperbound}
There exists a constant $c<\infty$ depending only on $d$, $\rho$ and $\alpha$ such that
\begin{align}
\pvis (r)& \lesssim  cr^{2(d-1)} f(r),\ d\ge 3.\label{e.bounds6}
\end{align}
\end{prop}

\begin{proof}
%The lower bound has already been proved in Proposition \ref{lowerbound}. Hence it only remains to prove the upper bound.
%Let
%\[
%D(x,y,s) = \{ \omega \in \Om : \not \exists w\in \omega \ :\ [x,y]_{s}\setminus [w]^\rho 
%\text{ consists of two disjoint components } \}
%
%and let $Q(x,y,s)$ be the event that $[x,y]_s \subset \V$. Moreover,
 
Fix $r>0$, $x,y\in {\mathbb R}^d$ and $\epsilon\in (0,\rho)$. Let $M(x,y,\epsilon) = \omega_\alpha(W^{*}_{[x,y]^{\rho -\epsilon}} )$ and let $A(x,y,\epsilon)$ be the event that there is a connected component of $[x,y]^{\epsilon} \cap \V $ that intersects both $B(x,\epsilon)$ and $B(y,\epsilon)$. Observe that on the event that $M(x,y,\epsilon)\ge 1$, there is some $z\in [x,y]$ such that $d(z,\bi )\le \rho-\epsilon$. For this $z$, we have $B(z,\epsilon)\subset \bi$. Any continuous curve $\gamma\subset [x,y]^{\epsilon}$ intersecting both $B(x,\epsilon)$ and $B(y,\epsilon)$ must also intersect $B(z,\epsilon)$. Hence, $\{M(x,y,\epsilon)\ge 1\}\subset A(x,y,\epsilon)^c$, and we get that

\begin{equation}\label{e.Aincl}
A(x,y,\epsilon) \subset \{M(x,y,\epsilon) = 0\}.
\end{equation}

Now we let
\begin{equation}\label{coveringnumber}
N(\epsilon,r) = \inf\left\{k \in \N: \exists x_1,x_2,..., x_k \in \partial B(r) \text{ such that } \bigcup_{i=1}^k B(x_i,\epsilon) \supset \partial B(r)\right\}
\end{equation}
be the covering number for a sphere of radius $r$, and note that $N(\epsilon,r) = O( (r/\epsilon)^{d-1} )$. For each $r>0$, let $(x_i)_{i=1}^{N(\epsilon,r)}$ be a set of points on $\partial B(r)$ such that $\partial B(r)\subset \cup_{i=1}^{N(\epsilon,r)} B(x_i,\epsilon)$. If $\{0 \geo \partial B(r)\}$ occurs there exists a $j \in \{1,2,...,N(\epsilon,r)\}$ such that $A(0,x_j,\epsilon)$ occurs. Hence, by the union bound and rotational invariance (Equation $2.28$ in~\cite{sznitman2013scaling}), 
\begin{align}\label{e.visest1}
\pvis(r) &\leq P\left(\bigcup_{i= 1}^{N(\epsilon,r)} A(0,x_i, \epsilon)\right) \le  N(\epsilon,r) P(A(0,x_1,\epsilon))\nonumber\\
 &\stackrel{~\eqref{e.Aincl}}{\leq}  O( (r/\epsilon)^{d-1})  \Pm(M(0,x_1,\epsilon)=0) .
\end{align}
Fix $x\in S^{d-1}$ and let $K_1=K_1(r,\rho) = [0,rx]^{\rho}$ and  $K_2=K_2(r,\rho,\epsilon)= [0,rx]^{\rho-\epsilon}$. Then 
\[
f(r) = \e^{-\alpha \capac(K_1)}
\]
and
\[
 \Pm(M(0,x_1,\epsilon)=0)  = \e^{-\alpha \capac(K_2)}.
\]
Hence,
\begin{equation}\label{e.visest2}
  \Pm(M(0,x_1,\epsilon)=0) =f(r)  \e^{\alpha(\capac(K_1) -\capac(K_2)) }
\end{equation}

We will now let $\epsilon=\epsilon(r)=1/r$ for $r \ge \rho^{-1}$ and show that 
\begin{equation}\label{e.o1eq}
\capac(K_1) -\capac(K_2) =O(1),\ r \to \infty.
\end{equation}

Let $((B_i(t))_{t\ge 0})_{i\ge 1}$ be a collection of i.i.d.\ processes with distribution $P_{\tilde{\e}_{K_1}}$ where $\tilde{\e}_{K_1}=\e_{K_1}/\capac(K_1).$
Recall that $[B_i]$ stands for the trace of $B_i$. Using the local description of the Brownian interlacements, see Equation \eqref{e.local}, we see that 

\begin{equation}\label{e.o2eq}
 \omega_\alpha(W^*_{K_1} \setminus W^*_{K_2}) \eqdist \sum_{i=1}^{N_{K_1}} {\bf 1}\{[B_i]\cap K_2=\emptyset\},
 \end{equation}
where $N_{K_1}$ is a Poisson random variable with mean $\alpha \capac(K_1)$ which is independent of the collection $((B_i(t))_{t\ge 0})_{i\ge 1}$, and the sum is interpreted as $0$ in case $N_{K_1}=0$. Taking expectations of both sides in~\eqref{e.o2eq} we obtain that
%to show~\eqref{e.o2eq} it suffices to prove that
%\begin{equation}\label{e.o3eq}
%\Pm \left(\bigcup_{i=1}^{N_{K_1}}\{ [B_i] \cap K_2 =\emptyset\} \right) \to 0, r \to \infty,
%\end{equation} 

\begin{align}\label{e.o5eq}
\alpha\nu(W^*_{K_1} \setminus W^*_{K_2}) &=E\left[\sum_{i=1}^{N_{K_1}} {\bf 1}\{[B_i]\cap K_2=\emptyset\}\right]\nonumber \\ & =E[N_{K_1}]P([B_1]\cap K_2=\emptyset)=\alpha \capac(K_1) P([B_1]\cap K_2=\emptyset),
\end{align}

where we used the independence between $N_{K_1}$ and $((B_i(t))_{t\ge 0})_{i\ge 1}$ and the fact the $B_i$-processes are identically distributed. 

Since $K_2\subset K_1$, it follows that
\begin{equation}\label{e.o4eq}
 \nu(W^*_{K_1} \setminus W^*_{K_2})=\capac(K_1) -\capac(K_2).
\end{equation}

From~\eqref{e.o5eq} and~\eqref{e.o4eq} it follows that

\begin{equation}\label{e.o3eq}
\capac(K_1) -\capac(K_2)=\capac(K_1) P([B_1]\cap K_2=\emptyset).
\end{equation}

Next, we find a useful upper bound on the last factor on the right hand side of~\eqref{e.o3eq}. Recall that for $t>0$ and $x\not \in B(0,t)$,

\begin{equation}\label{e.ballhitest}
P_x ( \tilde{H}_{B(0,t)} < \infty) = (t/|x|)^{d-2},
\end{equation}
see for example Corollary 3.19 on p.72 in \cite{morters2010brownian}.
%Now, by conditioning on $N_{K_1}=n \in \N$ and noting that $N_{K_1} \sim \text{Po}(\alpha \capac(K_1)$ we obtain 
%
%\begin{align*}
%\Pm \left(\bigcup_{i=1}^{N_{K_1}} [B_i] \cap K_2 = \emptyset \right) &= \sum_{n \geq 1} \Pm \left(\bigcup_{i=1}^{n} [B_i] \cap K_2 = \emptyset \right) \frac{\left( \alpha \capac(K_1) \right)^n}{n!} \e^{-\alpha \capac(K_1)}  \\ 
%&= \sum_{n \geq 1} \left(P_{e_{K_1}} \left( \tilde{H}_{K_2} = \infty \right) \right)^n  \frac{\left( \alpha \capac(K_1) \right)^n}{n!} \e^{-\alpha \capac(K_1)}.
%\end{align*}
Now,
\begin{equation}\label{e.hitintest}
P([B_1]\cap K_2=\emptyset)=P_{\tilde{\e}_{K_1}} \left( \tilde{H}_{K_2} = \infty \right) = \int_{\partial K_1} P_{y} ( \tilde{H}_{K_2} = \infty  ) \tilde{e}_{K_1} (\id y).
\end{equation}
For $z \in \partial K_1$ let $z' $ be the orthogonal projection of $z$ onto the line segment $[0,rx]$. Since $B(z',\rho-\epsilon)\subset K_2$ we have
\begin{equation}\label{e.hitincl}
\{\tilde{H}_{K_2} = \infty \}\subset \{\tilde{H}_{B(z',\rho-\epsilon)} = \infty \}.
\end{equation}
We now get that
\begin{align*}
P &\left( [B_1]\cap K_2=\emptyset\right)\stackrel{~\eqref{e.hitintest},~\eqref{e.hitincl}} {\leq} \int_{\partial K_1} P_{y} ( \tilde{H}_{B(y',\rho-\epsilon)} = \infty  ) \tilde{\e}_{K_1} (\id y) \\
& \stackrel{~\eqref{e.ballhitest}}{=} 1- \left( \frac{ \rho -\epsilon}{\rho} \right)^{d-2} = 1-(1-\epsilon/\rho)^{d-2} = O(1/r),
\end{align*}
where we recall that we made the choice $\epsilon=1/r$ for $r\ge \rho^{-1}$ above. Combining this with the fact that $\capac(K_1)=O(r)$ and~\eqref{e.o3eq} now gives~\eqref{e.o1eq}. Equations ~\eqref{e.visest1} and~\eqref{e.visest2} and~\eqref{e.o1eq} finally give that
\[
\pvis(r)  \leq O\left( r^{2(d-1)} \right) f(r)
\]
as $r \to \infty$. This establishes the upper bound in~\eqref{e.bounds4}
\end{proof}

 \section{Visibility for Brownian excursions in the unit disk}\label{s.BEproof}

In this section, we give the proof of Theorem~\ref{t.BEmain}. The method of proof we use here is an adaption of the method used in paper III of \cite{tykesson2008thesis}, which is an extended version of the paper \cite{benjamini2009visibility}. 
We first recall a result of Shepp \cite{shepp1972covering} concerning circle covering by random intervals. Given a decreasing sequence $(l_n)_{n\ge 1}$ of strictly positive numbers, we let $(I_n)_{n\ge 1}$ be a sequence of independent open random intervals, where $I_n$ has length $l_n$ and is centered at a point chosen uniformly at random on $\partial {\mathbb D}/(2\pi)$ (we divide by $2\pi$ since Shepps result is formulated for a circle of circumference $1$). Let $E:= \limsup_n I_n$ be the random subset of $\partial {\mathbb D}$ which is covered by infinitely many intervals from the sequence $(I_n)_{n\ge 1}$ and let $F:= E^c$. If $\sum_{n=1}^{\infty}l_n=\infty$ then $F$ has measure $0$ a.s.\ but one can still ask if $F$ is empty or non-empty in this case. Shepp \cite{shepp1972covering} proved that
\begin{thm}\label{Shepp}
$P(F= \emptyset) = 1$ if
\begin{equation}\label{e.sheppcond}
\sum_{n=1}^{\infty} \frac{1}{n^2} \e^{l_1+l_2+...+l_n} = \infty,
\end{equation}
and $P(F= \emptyset) = 0$ if the above sum is finite.
\end{thm}

Theorem~\ref{Shepp} is formulated for open intervals, but  the result holds the same if the intervals are taken to be closed or half-open, see the remark on p.$340$ of~\cite{shepp1972covering}.
 
A special case of Theorem~\ref{Shepp}, which we will make use of below, is that if $c>0$ and $l_n=c/n$ for $n\ge 1$, then (as is easily seen from~\eqref{e.sheppcond})

\begin{equation}\label{e.ccond}
 P(F=\emptyset)=1 \mbox{ if and only if }c\ge 1.
\end{equation}
 
 Before we explain how we use Theorem \ref{Shepp}, we introduce some additional notation. If $\gamma\subset \bar{{\mathbb D}}$ is a continuous curve, it generates a "shadow" on the boundary of the unit disc. The shadow is the arc of $\partial {\mathbb D}$ which cannot be reached from the origin by moving along a straight line-segment without crossing $\gamma$. More precisely, we define the arc ${\mathcal S}(\gamma)\subseteq \partial {\mathbb D}$ by 

$$ {\mathcal S}(\gamma)=\{e^{i \theta}\,:\,[0,\e^{i\theta})\cap \gamma\neq \emptyset\},$$
and let $\Theta(\gamma)=\mathrm{length}({\mathcal S}(\gamma))$, where $\mathrm{length}$ stands for arc-length on $\partial {\mathbb D}$.

 We now explain how we use Theorem~\ref{Shepp} to prove Theorem~\ref{t.BEmain}. First we need some additional notation. For $\omega=\sum_{i\ge 1}\delta_{(w_i,\alpha_i)}\in \Omega_{{\mathbb D}}$ and $\alpha>0$ we write $\omega_{\alpha}=\sum_{i\ge 1}\delta_{(w_i,\alpha_i)}1\{\alpha_i\le \alpha\}.$ Then under ${\mathbb P}_{{\mathbb D}}$, $\omega_{\alpha}$ is a Poisson point process on $W_{{\mathbb D}}$ with intensity measure $\alpha \mu$. Each $(w_i,\alpha_i)\in \omega_{\alpha}$ generates a shadow ${\mathcal S}(w_i)\subseteq \partial {\mathbb D}$ and a corresponding shadow-length $\Theta(w_i)\in [0,2\pi]$. The process of shadow-lengths 
 
 $$\Xi_{\alpha}:=\sum_{(w_i,\alpha_i)\in \mathrm{supp}(\omega_{\alpha})} \delta_{\Theta(w_i)}1\{\Theta(w_i)<2\pi\}$$ 
 
 is a non-homogeneous Poisson point process on $(0,2\pi)$, and we calculate the intensity measure of this Poisson point process below, see~\eqref{e.shadowproc}. Since Brownian motion started inside $\D$ stopped upon hitting $\partial \D$ has a positive probability to make a full loop around the origin, there might be a random number of shadows that have length $2\pi$ which we have thrown away in the definition of $\Xi_{\alpha}$. However, this number will be a Poisson random variable with finite mean (see the paragraph above~\eqref{e.fullshadow}), so those shadows will not cause any major obstructions.  Now, for $i\ge 1$, we denote by $\Theta_{(i),\alpha}$ the length of the $i$:th longest shadow in $\mathrm{supp}(\Xi_{\alpha})$.   We then show that 
 
 \begin{equation}\label{e.sumtoshow}
 \sum_{n=1} \frac{1}{n^2} \e^{(\Theta_{(1),\alpha}+\Theta_{(2),\alpha}+...+\Theta_{(n),\alpha})/(2\pi)} = \infty \mbox{ a.s. }
 \end{equation}
 
 if $\alpha \ge \pi /4$ and finite a.s.\ otherwise, from which Theorem~\ref{t.BEmain} easily will follow using Theorem~\ref{Shepp}. 
 
 We now recall some facts of one-dimensional Brownian motion which we will make use of. If $(B(t))_{t\ge 0}$ is a one-dimensional Brownian motion, its range up to time $t>0$ is defined as 
 
 $$R(t) = \sup_{s \leq t} B(s) - \inf_{s \leq t} B(s).$$
 
 The density function of $R(t)$ is denoted by $h(r,t)$ and we write $h(r)$ for $h(r,1)$. An explicit expression of $h(r,t)$ can be found in \cite{feller1951range}. The expectation of $R(t)$ is also calculated in \cite{feller1951range}. In particular, 
 
 \begin{equation}\label{e.rangexpectation}
 E[R(1)]= 2\sqrt{ \frac{2}{\pi} }.
 \end{equation}

Let $(B(t))_{t\ge 0}$ be a one-dimensional Brownian motion with $B(0)=a\in \R$. Let $H_{a}=\inf\{t\ge 0\,:\,B(t)=0\}$ be the hitting time for the Brownian motion of the value $0$. The density function of $H_a$ is given by

\begin{equation}\label{e.hittingdistr}
f_a(t) = |a| \e^{-a^2/2t}/ \sqrt{2 \pi t^3},\mbox{ }t\ge 0.
\end{equation}

Now let $W=(W(t))_{t\ge 0}$ be a two-dimensional Brownian motion with $W(0)=x\in  \D\setminus \{0\}$ stopped upon hitting $\partial {\mathbb D}$. Observe that the distribution of the length of the shadow generated by $W$, $\Theta(W)$, depends on the starting point $x$ only through $|x|$. The distribution of $\Theta(W)$ might be known, but since we could not find any reference we include a derivation, which is found in Lemma~\ref{l.shadowdistr} below. We thank K. Burdzy for providing a version of the arguments used in the proof of the lemma.

%Define the random arc ${\mathcal S}(B)\subset \partial {\mathbb D}$ by 
%
%$$ {\mathcal S}(B)=\{e^{i \theta}\,:\,L_{\infty}(\theta)\cap B\neq \emptyset\},$$
%
%and let $\Theta(B)=\mathrm{length}({\mathcal S}(B))$. 

% Let  $\phi(x) =\e^{-x^2/2}/ \sqrt{2\pi}$, $x\in {\mathbb R}$, be the density of a standard normal random variable.
 \begin{lem}\label{l.shadowdistr}
Suppose that $W=(W(t))_{t\ge 0}$ is a two-dimensional Brownian motion started at $x \in \D\setminus \{0\}$, stopped upon hitting $\partial \D$. Then, for $\theta \in (0,2\pi]$,
\begin{equation}
P( \theta\leq  \Theta(W) \leq 2\pi ) =  \int_{\{(r,t)\,:\,r \sqrt{t} \geq \theta\}} f_{\log(|x|)}(t) h(r) \id t \id r .
\end{equation}
%for $\theta \in [0,2 \pi ]$ where 
%\begin{equation}
%f_a(t) = |a| \e^{-a^2/2t}/ \sqrt{2 \pi t^3}
%\end{equation}
%is the hitting time distribution of a one-dimensional Brownian motion hitting the value $a \in \R $, and $h(r)$ is the density of $R(1) = \sup_{s \leq 1} B(s) - \inf_{s \leq 1} B(s)$.
\end{lem}
\begin{proof}
Without loss of generality, suppose that the starting point $x\in (0,1)$. We write $W(t)=s(t)\e^{\im \alpha(t)}$ where $\alpha(t)$ is the continuous winding number of $W$ around $0$, and $s(t)=|W(t)|$. Consider the process $(\phi(t))_{t\ge 0}$ living in the left half-plane defined by $\imag \phi(t)= \alpha(t)$ and $\re \phi(t) =\log(s(t))$.  For use below, we note that for $\beta\in [0,2\pi)$,
\begin{equation}\label{e.multifunc}
\log\left( \left\{ r \e^{\im \beta} : 0 < r < 1 \right\} \right) = \left\{ x+ \im (\beta+2\pi k) : x<0,\,k\in \Z \right\}.
\end{equation}
By conformal invariance of Brownian motion, the law of the image of $\phi$ is the same as the law of the image of a  two-dimensional Brownian $\tilde{W}=(\tilde{W}(t))_{t\ge 0}$ started at $\log(x) \in (-\infty,0)$ and stopped upon hitting $\{ z \in \C : \re z = 0\}$. 
%Now, first note that $P_x(\Theta = 2 \pi)>0$ since there is a positive probability that the Brownian motion has a winding number, $W$, greater or equal to $ 1$. 
Let 
\[
T= \inf \left\{ t >0 : \re \tilde{W}(t) = 0  \right\} \mbox{ and } R(t) = \sup_{s \leq t} \imag \tilde{W}(s) - \inf_{s \leq t} \imag  \tilde{W}(s).
\]
Using~\eqref{e.multifunc}, we see that whenever $\theta \in (0,2\pi]$,
\begin{equation}\label{e.confeq}
P( \theta \leq \Theta \leq 2 \pi   ) = P(  \theta  \leq   R(T) ).
\end{equation}
%and 
%\[
%P_x(  \Theta = 2 \pi   ) = P(     R(T) \geq   2 \pi ),
%\]
%since if a Brownian motion has a winding number $W \geq 1$ then the corresponding range will be greater than $2 \pi$. 
Moreover, $T$ and $\imag  \tilde{W}(t)$ are independent since $T$ is determined by $\re \tilde{W}(t)$ and $\re \tilde{W}(t)$ and $\imag \tilde{W}(t)$ are independent.  Since $T$ and $\imag \tilde{W}(t)$ are independent we have by Brownian scaling
\begin{equation}\label{e.distreq}
R(T) \eqdist \sqrt{T} R(1).
\end{equation}
Hence
\begin{equation}
P( \theta\leq \Theta \leq 2\pi)\stackrel{~\eqref{e.confeq},~\eqref{e.distreq}}{=} P(\sqrt{T} R(1)\ge \theta)= \int_{r \sqrt{t} \geq \theta} f_{|\log(x)|}(t) h(r) \id t \id r,
\end{equation}
finishing the proof of the lemma.
%where $h(r)$ is the density of $R(1)$ which is given by
%\[
%h(r) = \sqrt{2 /\pi} r^{-1}L'(r/2 \sqrt{t}) 
%\]
%where 
%\[
%L(z) = \frac{\sqrt{2 \pi}}{z} \sum_{k =1}^\infty \exp\left( -(2 k -1 )^2 \pi^2 /8z^2 \right)
%\]
%is the Kolmogorov-Smirnoff distribution function, see Equations 3.6-3.8 \cite{feller1951}.
\end{proof}   
%For $w \in W_\D$ let 
% ${\mathcal S}(w)=\{e^{i \theta}\,:\,L_{\infty}(\theta)\cap w\neq \emptyset\}$ and let $\Theta(w)=\mathrm{length}({\mathcal S}(w))$. 
 
In the next lemma, we calculate the intensity measure of $\Xi_{\alpha}$. For $\theta\in (0,2\pi]$ define 
\begin{equation}
A_\theta = \left\{ w \in W_\D : \theta \leq \Theta(w)\right\}.
\end{equation}

 \begin{lem}\label{l.ameasure}
 For $\theta \in (0,2 \pi]$
\begin{equation}
\mu( A_\theta  ) = \frac{8}{  \theta}.
\end{equation}
 \end{lem}
\begin{proof}
%The Brownian excursion measure is defined as the limit
%\[
%\lim_{\epsilon \downarrow 0} \frac{1}{\epsilon }P_{\sigma_{1-\epsilon}} ( \cdot)
%\]
By the definition of $\mu$, we must show that
\[
\lim_{\epsilon\downarrow 0} \frac{2 \pi}{\epsilon }P_{\sigma_{1-\epsilon}} ( \theta \leq \Theta   )=\frac{8}{  \theta}.
\]
%First of all, it suffices to consider all $\epsilon< \epsilon'$ for some $\epsilon'>0$ arbitrary but fixed. 

We now get that
\begin{align*}
\frac{2 \pi }{\epsilon} P_{\sigma_{1-\epsilon}} ( \theta \leq \Theta) & = \frac{2 \pi}{\epsilon} \int_{\partial B(0,1-\epsilon)} P_z( \theta \leq \Theta) \sigma_{1-\epsilon} ( \id z)\\ &=  \frac{2 \pi}{\epsilon} P_{1-\epsilon}( \theta \leq \Theta)  = \frac{2 \pi}{\epsilon}\int_{r \sqrt{t} \geq \theta} f_{|\log(1-\epsilon)|}(t) h(r) \id t \id r,
\end{align*}
where we used rotational invariance in the second equality and Lemma~\ref{l.shadowdistr} in the last equality. We have

\begin{align*}
 \frac{2 \pi}{\epsilon}\int_{r \sqrt{t} \geq \theta} f_{|\log(1-\epsilon)|}(t) h(r) \id t \id r \stackrel{~\eqref{e.hittingdistr}}{=}  \frac{-\log(1-\epsilon)}{\epsilon}\int_{r \sqrt{t} \geq \theta}  \e^{-\log^2(1-\epsilon)/2t} \sqrt{ \frac{ 2 \pi}{ t^3} }   h(r) \id t \id r.
\end{align*}

%Now, let 
%\[
%g_\epsilon(t) = \frac{2 \pi}{\epsilon} f_{|\log (1-\epsilon) |}(t) = \frac{-\log(1-\epsilon)}{\epsilon \sqrt{2 \pi t^3}} \e^{-\log^2(1-\epsilon)/2t}/
%\]
%and note that
%\[
%\lim_{\epsilon \downarrow 0} g_\epsilon (t) = \frac{1}{\sqrt{2\pi t^3}}.
%\]
Note that $-\log(1-\epsilon)/ \epsilon  \to 1$ as $\epsilon \to 0$, $\e^{-\log^2(1-\epsilon)/2t}$ is monotone increasing in $\epsilon$, and $\e^{-\log^2(1-\epsilon)/2t}\to 1$ as $\epsilon\to 0$ for $t>0$. Hence, the monotone convergence theorem gives that
\begin{equation}
\mu( A_\theta ) = \lim_{\epsilon \downarrow 0} \frac{2 \pi}{\epsilon}  \int_{r \sqrt{t} \geq \theta} f_{|\log(1-\epsilon)|}(t) h(r) \id t \id r \stackrel{~\eqref{e.hittingdistr}}{=}\int_{r \sqrt{t} \geq \theta }\sqrt{ \frac{2 \pi}{ t^3}} h(r) \id r \id t.
\end{equation}
This integral is easily computed as
\begin{align}
 \int_{r \sqrt{t} \geq \theta} \sqrt{ \frac{2 \pi }{ t^3} } h(r) \id r \id t & =\sqrt{2 \pi} \int_0^\infty \int_{t \geq (\theta/r)^2} \frac{1 }{t^{3/2}}  \id th(r) \id r\\ 
 & = \sqrt{ 2\pi } \frac{2}{\theta} \int_0^\infty r h(r) \id r=  \sqrt{ 2\pi } \frac{2}{\theta} E[R(1)]\stackrel{~\eqref{e.rangexpectation}}{=} \frac{8}{ \theta}, 
\end{align}
%Where in the last equality we utilized equation 1.4 in \cite{feller1951}:
%\[
%E[R(1)] = 2 \sqrt{2/\pi}
%\]
finishing the proof of the lemma.
\end{proof}

\remark Lemma~\ref{l.ameasure} implies that $\mu(A_{2\pi})=4/\pi $. Hence, under ${\mathbb P}_{\D}$, $\omega_{\alpha}(A_{2\pi})$ is a Poisson random variable with mean $\alpha 4/\pi $. In particular,
\begin{equation}\label{e.fullshadow}
{\mathbb P}_{\D}(\omega_{\alpha}(A_{2\pi})=0)>0.
\end{equation}
\noindent
We will now use Lemma~\ref{l.ameasure} to prove Theorem~\ref{t.BEmain}.

\noindent
{\it Proof of Theorem~\ref{t.BEmain}}. Define the measure $m$ on $(0,2 \pi)$ by letting 
\[
m(A) =  \int_A \frac{8}{ t^2} \id t,\ A \in \B((0,2\pi)). %\mbox{ and }m(\{2\pi\})=2/\pi^2.
\]

\noindent
Lemma~\ref{l.ameasure} implies that under ${\mathbb P}_{{\mathbb D}}$, 

\begin{equation}\label{e.shadowproc}
\Xi_{\alpha}\mbox{ is a Poisson point process on }(0,2\pi)\mbox{ with intensity measure }\alpha m.
\end{equation}
\noindent
We now consider the Poisson point process on $(1/(2\pi),\infty)$ defined by

\begin{equation}\label{e.invproc}
\Xi_\alpha^{-1}:=\sum_{(w_i,\alpha_i)\in \mathrm{supp}(\omega_{\alpha})} \delta_{\Theta(w_i)^{-1}}.
\end{equation}
\noindent
Now note that for $ 1/(2 \pi)< t_1 < t_2$ we have that 
\[
m([1/t_2 , 1/t_1]) = 8  ( t_2 -t_1).
\] 
\noindent
Hence, $\Xi_\alpha^{-1}$ is a homogeneous Poisson point process on $(1/(2\pi),\infty)$ with intensity $8\alpha $. Now let $\Delta_1=1/\Theta_{(1),\alpha}$ and for $n\ge 2$ let

$$\Delta_n:=1/\Theta_{(n),\alpha}-1/\Theta_{(n-1),\alpha}.$$
\noindent
Then $\Delta_{n}$ is a sequence of i.i.d.\ exponential random variables, with mean $1/(8 \alpha)$.  Since $$1/\Theta_{(n),\alpha} = \sum_{i=1}^n \Delta_i,$$ we get that

\begin{equation}\label{e.infofteq1}
 P\left(\left|\frac{1}{\Theta_{(n),\alpha}}- \frac{n  }{8 \alpha}\right | > n^{3/4} \text{ i.o. }\right) = 0.
\end{equation}
\noindent
Since

\[
 \left| \Theta_{(n),\alpha} - \frac{8 \alpha}{n  }\right| = \left |\frac{1/\Theta_{(n),\alpha} - n  /(8 \alpha)}{n /(8 \alpha  \Theta_{(n),\alpha})} \right|,
  \]
and $1/\Theta_{(n),\alpha}> c n$ for all but finitely many $n$ for some constant $c>0$ a.s., 
Equation~\eqref{e.infofteq1} implies that for some constant $c'(\alpha)<\infty$,

\begin{equation}\label{e.errsummable}
P\left(\left| \Theta_{(n),\alpha} - \frac{8 \alpha}{n  }\right|\le c'(\alpha) n^{-5/4}\mbox{ for all but finitely many }n \right)=1.
\end{equation}

Let $Y_{n,\alpha}=\sum_{i=1}^n  \Theta_{(i),\alpha}-\sum_{i=1}^n \frac{8 \alpha}{i }$. From~\eqref{e.errsummable} and the triangle inequality we see that a.s., $Y_{\infty,\alpha}:=\lim_{n\to\infty} Y_{n,\alpha}$ exists and $|Y_{\infty,\alpha}|<\infty$ a.s.\ Hence,  
\[ 
\tilde{Y}_{\infty,\alpha}:=  \lim_{n\to \infty}\left(\sum_{i=1}^n\frac{\Theta_{(i),\alpha}}{2\pi}-\frac{4 \alpha \log(n)}{\pi }\right)
\]
 exists and is finite a.s. Hence, the sum in~\eqref{e.sumtoshow} is finite a.s.\ if $\alpha<\pi /4$ and infinite a.s.\ if $\alpha\ge \pi /4$. Let $\tilde{V}_\infty^{\alpha}$ denote the event that there is some $\theta\in [0,2\pi)$ such that $[0,\e^{i\theta})$ intersects only a finite number of trajectories in the support of $\omega_{\alpha}$. The above, together with~\eqref{e.fullshadow}, shows that ${\mathbb P}_\D(\tilde{V}_\infty^{\alpha})=1$ if $\alpha<\pi /4$ and ${\mathbb P}_\D(\tilde{V}_\infty^{\alpha})=0$ if $\alpha\ge \pi /4$. It remains to argue that ${\mathbb P}_{\D}(V_\infty^{\alpha})>0$ when $\alpha<\pi/4$. So now fix $\alpha<\pi/4$. Let $\tilde{V}_{\infty,R}^{\alpha}$ be the event that there is some $\theta\in [0,2\pi)$ such that $[0,\e^{i\theta})$ intersects only trajectories in the support of $\omega_{\alpha}$ which also intersect the ball $B(0,R)$. If $\tilde{V}_\infty^{\alpha}$ occurs, then for some random $R_0<1$, the event $\tilde{V}_{\infty,R}^{\alpha}$ occurs for every $R\in(R_0,1)$. Hence for some $R_1<1$, ${\mathbb P}_\D(\tilde{V}_{\infty,R_1}^{\alpha})>0.$ Suppose that $\bar{\omega}\in \Omega_{\D}$ and write 

$$\hat{\omega}_{\alpha}= 1_{W_{B(0,R_1)}}\bar{\omega}_\alpha+1_{W_{B(0,R_1)}^c}\omega_\alpha.$$

Observe that if $\omega_{\alpha}\in \tilde{V}_{\infty,R_1}^{\alpha}$ and $\bar{\omega}_{\alpha}(W_{B(0,R_1)})=0$, then $\hat{\omega}_{\alpha}\in V_\infty^{\alpha}$. Hence 
$${\mathbb P}_{\D}^{\otimes 2}(\hat{\omega}_{\alpha}\in V_\infty^{\alpha})\ge {\mathbb P}_{\D}(\tilde{V}_{\infty,R_1}^{\alpha}){\mathbb P}_{\D}(\bar{\omega}_{\alpha}(W_{B(0,R_1)})=0)>0.$$

The result follows, since $\omega_{\alpha}$ under ${\mathbb P}_{\D}$ has the same law as $\hat{\omega}_{\alpha}$ under ${\mathbb P}_{\D}^{\otimes 2}$.\qed

\noindent \\ {\bf Acknowledgements: }We thank Johan Jonasson for comments on parts of the paper. We thank 
Krzysztof Burdzy for providing a version of the argument of Lemma~\ref{l.shadowdistr}.

\newpage

    				%%%%BIBLIOGRAPHY %%%%
\bibliography{references} % references.bib
\end{document}